\documentclass[12pt]{article}
\usepackage[dvips]{epsfig}
\usepackage{amsmath,amssymb,euscript,graphicx,amsfonts,verbatim}
\usepackage{enumerate,color}
\usepackage{graphicx}
\usepackage{xcolor}

\usepackage[colorlinks=true,linkcolor=blue,urlcolor=blue,citecolor=blue]{hyperref}
\usepackage{enumerate}

\newtheorem{theorem}{Theorem}[section]

\newtheorem{lemma}[theorem]{Lemma}
\newtheorem{corollary}[theorem]{Corollary}

\newenvironment{proof}{{\noindent \sc Proof.}}{\hfill $\Qed$\\}

\newcommand{\Qed}{\rule{2.5mm}{3mm}}

\DeclareMathOperator\gl{GL}

\usepackage{tablefootnote}
\usepackage{float}

\newcounter{case}

\renewcommand{\thecase}{\arabic{case}}

\newcounter{subcase}

\numberwithin{subcase}{case}

\begin{document}


\begin{center}
{\bf\Large On the Terwilliger algebra of the group association scheme of $C_n \rtimes C_2$} \\ [+4ex]
Roghayeh Maleki {\small$^{1,2}$}
\\ [+2ex]
{\it \small 
$^1$University of Primorska, UP IAM, Muzejski trg 2, 6000 Koper, Slovenia\\
$^2$University of Primorska, UP FAMNIT, Glagolja\v ska 8, 6000 Koper, Slovenia
}
\end{center}

\begin{abstract}
	In 1992, Terwilliger introduced the notion of the \emph{Terwilliger algebra} in order to study association schemes. The Terwilliger algebra of an association scheme $\mathcal{A}$  is the subalgebra of the complex matrix algebra, generated by the \emph{Bose-Mesner algebra}  of $\mathcal{A}$ and its dual idempotents with respect to a point $x$.
	
	In [{\em Kyushu Journal of Mathematics}, 49(1):93--102, 1995]  Bannai and Munemasa determined the dimension of the Terwilliger algebra of abelian groups and dihedral groups, by showing that they are triply transitive (i.e.,  triply regular and dually triply regular). In this paper, we give a generalization of their results to the group association scheme of semidirect products of the form $C_n\rtimes C_2$, where $C_m$ is a cyclic group of order $m\geq 2$. Moreover, we will give the complete characterization of the Wedderburn components of the Terwilliger algebra of these groups.
\end{abstract}

\begin{quotation}
	
\end{quotation}

\begin{quotation}
	\noindent 
	{\em Math. Subj. Class.:}  05E30, 05C50.
\end{quotation}


\section{Introduction}

	Association schemes are combinatorial objects that were introduced by Bose and Shimamoto in 1952 \cite{bose1952classification} to study partially balanced incomplete block designs. Since then these objects have been used to study certain combinatorial objects such as graphs and permutation groups from a more algebraic viewpoint. In 1992, Terwilliger published his seminal papers \cite{PTerwilliger} on the so-called \emph{subconstituent algebra} of an association scheme. This algebra is known nowadays as the \emph{Terwilliger algebra}.

	Given a finite non-empty set $X = \{1,2,\ldots,n\}$  and a set of relations $\mathcal{R} = \{ R_0,R_1,\ldots,R_t \}$ on $X$ (i.e., subsets of $X\times X$), the pair $\mathcal{A} = (X,\mathcal{R})$ is called an \emph{association scheme} if the following conditions hold.
\begin{enumerate}[(i)]
	\item $R_0 = \{ (x,x) \mid x\in X \}$,
	\item $\mathcal{R}$ is a partition of $X \times X$,
	\item for any $i\in \{0,1,\ldots,t\}$, the relation $R_i^T := \{ (y,x) \mid (x,y) \in R_i \}$ belongs to $\mathcal{R}$,
	\item  for all $i,j,k\in \{0,1,\ldots,t\}$ and for all $(x,y) \in R_k$, the number $p_{ijk} := |\{ z\in X \mid (x,z)\in R_i \mbox{ and } (z,y)\in R_j \}|$ depends only on the choice of $i,j$, and $k$.
\end{enumerate}

The  numbers $p_{ijk}$ are known as \emph{intersection numbers} for the association scheme $\mathcal{A}$. We can reformulate this definition of association schemes in terms of matrices. For any $i\in \{ 0,1,\ldots,t \}$, define the $01$-matrix $A_i$ to be the $n\times n$ matrix indexed in its rows and columns by $X$ and with entry $A_i(u,v) = 1$ if $(u,v)\in R_i$, and $0$ otherwise. Now, we can refine the above definition of association schemes as follows. If $\mathcal{A} = \{ A_0,A_1,\ldots,A_t \}$, then we say that $\mathcal{A}$ is an \emph{association scheme} if 
\begin{enumerate}[(i)]
	\item $A_0$ is equal to the identity matrix,
	\item $A_0 + A_1 + \ldots + A_t = J$, where $J$ is the matrix of all ones,
	\item for any $i\in \{0,1,\ldots,t\}$, $A_i^T \in \mathcal{A}$, and
	\item for any $i,j\in \{0,1,\ldots,t\}$, we have $A_iA_j = \displaystyle \sum_{k=0}^t p_{ijk} A_k$.
\end{enumerate}

Association schemes have been extensively studied in the literature \cite{bannai1991subschemes,delsarte1998association,godsil2017algebraic,martin2009commutative,zieschang2006algebraic,zieschang2005theory}. 
Let $\mathcal{A} = \{A_0,A_1,\ldots,A_t\}$ be an association scheme on the set $X= \{1,2,\ldots,n\}$. Fix an element $x\in X$. For any $i\in \{0,1,\ldots,t\}$, we define $E_{i,x}^*$ to be the $n\times n$ diagonal matrix whose $(u,u)$ entry is given by $E_{i,x}^*(u,u) = A_i(x,u)$. That is, the diagonal of $E_{i,x}^*$ is exactly the $x$-th row of $A_i$. The \emph{Terwilliger algebra} of the association scheme $\mathcal{A}$ with respect to $x$ is the subalgebra of $M_{n}(\mathbb{C})$ generated by $\{ A_0,A_1,\ldots A_t \} \cup \{ E_{0,x}^*,E_{1,x}^*,\ldots,E_{t,x}^* \}$. We shall denote this algebra by $T(x)$.

Of interest to us are Terwilliger algebras that arise from groups. Let $G$ be a finite group and let $\mathcal{C} = \{ {\rm Cl}_0,{\rm Cl}_1,\ldots,{\rm Cl}_t \}$ be the collection of all conjugacy classes of $G$ (assume that ${\rm Cl}_0=\{1\}$). For any $i\in \{0,1,\ldots,t\}$, we define $R_i:= \{ (x,y) \in G\times G \mid yx^{-1} \in {\rm Cl}_i \}$. It is not hard to verify that the pair $(G,\{R_0,R_1,\ldots,R_t\})$ forms an association scheme. This association scheme is known as the \emph{group association scheme} of $G$. In 1995, Bannai and Munemasa studied the Terwilliger algebra of abelian groups and dihedral groups \cite{bannai1995terwilliger}. In \cite{balamaceda1994terwilliger}, Balmaceda and Oura found the structure of the Terwilliger algebra for the symmetric group $S_5$ and the alternating group $A_5$. A similar work by Hamid and Oura for $\operatorname{PSL}(2,7),\ A_6, \mbox{ and } S_6$ appeared in \cite{hamid2019terwilliger}. Recently, Bastian \cite{bastian2021terwilliger} gave more theoretical and computational results on the Terwilliger algebra of group association schemes. In this paper, we study the Terwilliger algebra of the group association scheme of the group with presentation
\begin{align*}
	D_{n,s}:= C_n \rtimes_s C_2 = \left\langle a,b \mid a^n = b^2 = 1,\ bab^{-1} = a^s  \right\rangle,
\end{align*}
where $s \geq 1$ is an integer such that $s^2 \equiv 1 \ (\operatorname{mod}{n})$ and $s \not\equiv 1 \ (\operatorname{mod}{n})$.

We will only consider the Terwilliger algebra with respect to the identity element of the group. That is, if $G$ is a group, then by the Terwilliger algebra of $G$, we mean the Terwilliger algebra $T(e)$, where $e\in G$ is the identity of $G$. Henceforth, the Terwilliger algebra of $G$ is denoted by $T(G)$ and the generators $(E_{i,e}^*)_{i\in \{ 0,1,\ldots,t \}}$ of $T(G)$ are denoted by $(E_i^*)_{i\in \{0,1,\ldots,t\}}$. Our first main result is stated as follows, which will be established in Section~\ref{sec:dim}.
\begin{theorem}
	Let $n,s\in \mathbb{N}$ such that $s\leq n$, $\gcd(n,s)=1$ and let $\tau:= \gcd(s-1,n)$. The dimension of $T(D_{n,s})$ is equal to $ \frac{n^2+3n\tau+4\tau^2}{2}$.\label{thm:main1}
\end{theorem}

Since $T(G)$ is a self-adjoint (i.e., closed under complex conjugate transposition) subalgebra of $M_{|G|}(\mathbb{C})$ for any finite group $G$, we know that $T(G)$ is \emph{semisimple}. In \cite[Appendix~B]{bastian2021terwilliger}, several open questions on the Wedderburn decomposition of Terwilliger algebras for group association schemes were asked. Our second result determines the Wedderburn decomposition of the Terwilliger algebra of the group association scheme of $D_{n,s}$. In particular, this answers all the questions in \cite{bastian2021terwilliger} for the dihedral group $D_{n,-1}$. We will give these results in Section~\ref{sec:wedd}.


\section{Background results}
Throughout this section, we let $G$ be a finite group and $T(G)$ be its Terwilliger algebra. Moreover, we let ${\rm Cl}_0 = \{1\}, {\rm Cl}_1,\ldots, {\rm Cl}_t$ be the conjugacy classes of $G$.

\subsection{Bounds on the dimension of $T(G)$}
Let $\mathcal{A} = \{ A_0,A_1,\ldots,A_t \}$ be an association scheme on the set $\{1,2,\ldots,n\}$. The \emph{Bose-Mesner algebra} of $\mathcal{A}$ is the subalgebra of $M_{n}(\mathbb{C})$ generated by $\mathcal{A}$. It is well known that the matrices in $\mathcal{A}$ form a basis as a $\mathbb{C}$-vector space for the Bose-Mesner algebra of the association scheme $\mathcal{A}$. In particular, the multiplication of two elements in $\mathcal{A}$ can be determined using the intersection numbers $(p_{ijk})_{i,j,k\in \{0,1,\ldots,t\}}$, which satisfy
\begin{align*}
	A_iA_j &= \sum_{k=0}^t p_{ijk}A_k.
\end{align*}	

The first result on the dimension of the Terwilliger algebra of a group association scheme that we present is related to the intersection numbers. Define the $\mathbb{C}$-vector space 
\begin{align*}
	T_0(G) = \operatorname{Span}_{\mathbb{C}}\{ E_i^*A_jE_k^* \mid 0\leq i,j,k \leq t \}.
\end{align*}
It is clear that $T_0(G)$ is always contained in $T(G)$. Hence, $\dim_{\mathbb{C}}T_0(G)\leq \dim_{\mathbb{C}}T(G)$. The following lemma gives the dimension of $T_0(G)$.
\begin{lemma}[{\cite[Lemma 1]{bannai1995terwilliger}}] 
	We have
	\begin{align*}
		\dim_{\mathbb{C}} T_0(G) = \left|\left\{ (i,j,k) \mid  p_{ijk} \neq 0, \,\, \mbox{for}\,\, i,j,k\in \{1,2,\ldots,t\}  \right\} \right|.
	\end{align*}
\end{lemma}
We can further refine the above result. For any $0\leq i,j,k\leq t$, we define $$S_{ijk} := \{ (g,h) \in {\rm Cl}_i\times {\rm Cl}_j \mid gh \in {\rm Cl}_k \}.$$ It is not hard to see that the set $S_{ijk} = \varnothing$ if and only if $p_{ijk} = 0$. Therefore, we derive the following lemma about the dimension of $T_0(G)$.

\begin{lemma}[\cite{bastian2021terwilliger}]
	$\dim_{\mathbb{C}}T_0(G) = | \left\{ (i,j,k) \mid {\rm Cl}_k \subseteq {\rm Cl}_i{\rm Cl}_j \right\} |$.\label{lem:lower}
\end{lemma}

Next, we present an upper bound on the dimension of $T(G)$ that is purely algebraic. Recall that $G$ acts on itself by conjugation. This action is intransitive and its orbits corresponds to the conjugacy classes of $G$. The stabilizer of $x\in G$ of this action is equal to the centralizer $C_G(x)$ of $x$ in $G$. It is obvious that $G$ acts on $G \times G$ by componentwise conjugation. Let $\ell$ be the number of orbits of this induced action of $G$ and let  $\mathcal{O}_0,\mathcal{O}_1,\ldots,\mathcal{O}_\ell$ be the orbits. For any $i\in \{0,1,\ldots,\ell\}$, we define the $|G|\times |G|$ matrices $B_i$ indexed by the group elements in its rows and columns such that
\begin{align*}
	B_i(g,h) =
	\begin{cases}
		1 \hspace{1cm} & \mbox{ if } (g,h)\in \mathcal{O}_i,\\
		0 & \mbox{ otherwise}. 
	\end{cases}
\end{align*}

The \emph{centralizer algebra} of $G$  is the subalgebra of $M_{|G|}(\mathbb{C})$ defined by
\begin{align*}
	\widetilde{T}(G) = \left\{ A\in M_{|G|}(\mathbb{C}) \mid  B_iA = A B_i, \mbox{ for } i\in \{0,1,\ldots,\ell\} \right\}.
\end{align*}

As shown in \cite{bannai1995terwilliger}, we have $T(G) \subseteq \widetilde{T}(G)$. Therefore, $\dim_{\mathbb{C}} T(G) \leq \dim_{\mathbb{C}}\widetilde{T}(G)$. By \cite{bannai1995terwilliger} again, we have that $\dim_{\mathbb{C}}\widetilde{T}(G) = \ell$, which is the number of orbitals of $G$. Recall that the number of orbits of $G$ acting on $G\times G$ by conjugation can be easily computed using the Orbit Counting Lemma. Let $x_0,x_1,\ldots,x_t$ be a system of representatives of the conjugacy classes ${\rm Cl}_0,{\rm Cl}_1,\ldots,{\rm Cl}_t$. The next result was proved by Bannai and Munemasa in \cite{bannai1995terwilliger}.
\begin{lemma}
	$\dim_{\mathbb{C}}T(G) \leq \displaystyle\sum_{i=0}^t |C_G(x_i)|$.\label{lem:upper}
\end{lemma}

If $G$ is a finite group for which the lower bound in Lemma~\ref{lem:lower} and the upper bound in Lemma~\ref{lem:upper} coincide, then we say that $G$ is \emph{triply transitive.}

\subsection{Wedderburn decomposition}\label{subsec:weddecth}

Recall that $T(G)$ is a semisimple algebra. By the Wedderburn theory of semisimple algebras,  $T(G)$ can be written as a direct sum of simple algebras. In this subsection, we present the tools to find these simple summands of $T(G)$. If $G$ is triply transitive, then we have $T(G) = \widetilde{T}(G)$. In this case, the Wedderburn decomposition of $T(G) = \widetilde{T}(G)$ can be determined from the character table of the group $G$, as we will see below. 

Since $\widetilde{T}(G)$ is semisimple, there exist primitive central idempotents $e_0,e_1,\ldots,e_t$ such that $$\widetilde{T}(G) = \bigoplus_{i=0}^t \widetilde{T}(G) e_i.$$
In fact, we have an explicit expression for these idempotents as follows.\\ Consider the representation $\mathfrak{X}: G \to \gl_{|G|}(\mathbb{C})$ such that for any $g\in G$, the matrix $\mathfrak{X}(g)$ is defined by
\begin{align*}
	\mathfrak{X}(g)_{u,v} = 
	\begin{cases}
		1 \hspace{1cm} & \mbox{ if } v = gug^{-1},\\
		0 \hspace{1cm} & \mbox{ otherwise},
	\end{cases}
\end{align*}
for any $u,v \in G$. The representation $\mathfrak{X}$ is called the \emph{permutation representation} of the action of $G$ by conjugation on itself. Let $\operatorname{Irr}(G) = \{ \chi_0,\chi_1,\ldots,\chi_t \}$ be the set of all non-equivalent irreducible characters of $G$. Then, for $i\in \{0,1,\ldots,t\}$, we have
\begin{align*}
	e_i = \frac{\chi_i(1)}{|G|} \sum_{g \in G} \overline{\chi_i(g)} \mathfrak{X}(g).
\end{align*}

For each $i\in\{0,1,\ldots,t\}$, $\widetilde{T}(G)e_i$ is isomorphic to $M_{d_i}(\mathbb{C})$ (complex matrix algebra of dimension $d^2_i$), for some non-negative integer $d_i$. It was proved in \cite{balamaceda1994terwilliger} that the numbers $d_i$ can be determined from the representation $\mathfrak{X}$ or equivalently, from its character. 
Let $\pi$ be the character corresponding to the permutation representation $\mathfrak{X}$. 
\begin{lemma}[\cite{balamaceda1994terwilliger}]
	For any $i\in \{0,1,\ldots,t\}$, the number $d_i$ is equal to the multiplicity of the irreducible character $\chi_i$ in $\pi$. That is, 
	\begin{align*}
		\pi &= \sum_{i=0}^t d_i \chi_i.
	\end{align*}\label{lem:wedder}
\end{lemma}
We note that for any $g\in G$, $\pi(g) = \left| \left\{ x\in G \mid gxg^{-1} =x \right\} \right| = |C_G(x)| $. Let $x_0,x_1,\ldots,x_t$ be a system of representatives of the conjugacy classes ${\rm Cl}_0,{\rm Cl}_1,\ldots,{\rm Cl}_t$. It is clear from the above lemma that 
\begin{align*}
	d_i &= \langle \pi ,\chi_i \rangle \\&= \frac{1}{|G|} \sum_{g \in G} \pi(g) \overline{\chi_i(g)}\\
	&= \frac{1}{|G|} \sum_{j=0}^t  |{\rm Cl}_j|\ |C_G(x_j)|\ \overline{\chi_i(x_j)}\\
	&= \sum_{j=0}^t \overline{\chi_i(x_j)}.
\end{align*}
Therefore, the number $d_i$ is equal to the row sum corresponding to $\chi_i$ in the character table of $G$, for any $i\in \{0,1,\ldots,t\}$.


\section{Dimension of $T(D_{n,s})$}\label{sec:dim}

To prove Theorem~\ref{thm:main1} we use Lemma~\ref{lem:lower} and Lemma~\ref{lem:upper}. We prove that the lower and the upper bounds given in these two lemmas coincide. Indeed, since $\dim_{\mathbb{C}}T_0(G)\leq \dim_{\mathbb{C}}T(G)$ and  $\dim_{\mathbb{C}} T(G) \leq \dim_{\mathbb{C}}\widetilde{T}(G)$, it is enough to show that $\dim_{\mathbb{C}}T_0(G)=\dim_{\mathbb{C}}\widetilde{T}(G)$.\\

Now, fix $n \geq 3$ and let $s \geq 1$ be such that $s^2 \equiv 1 \ (\operatorname{mod}{n})$ and $s \not\equiv 1 \ (\operatorname{mod}{n})$. We consider the semidirect product of $C_n$ and $C_2$ given by 
$$
D_{n,s}=\langle a,b \mid a^n=1, b^2=1, bab^{-1}=a^s \rangle.
$$ From now on, to simplify, let $G=D_{n,s}$ and set $\tau={\rm gcd}(s-1,n)$. The subgroup of $C_n = \langle a\rangle$ of order $\tau \mid n$ is given by: 
$$C_n[\tau]=\{ g \in C_n \mid g^{s-1}=1 \}=\{ g \in C_n \mid g^\tau=1 \}.$$ Also, the subgroup of $C_n=\langle a\rangle$ of order $n/\tau$ is given by: 
$$C_n^\tau=\{g^{s-1} \mid g \in C_n \}=\{ g^\tau \mid g \in C_n\}.$$


\subsection{Conjugacy classes and centralizers of elements in $G$}
To begin with, we determine all conjugacy classes and centralizers of elements in $G$. Recall that $\tau={\rm gcd}(s-1,n)$. For the remainder of this section, we let ${\rm Cl}(u) := \{ gug^{-1} \mid g\in G \}$ be the conjugacy class containing the element $u \in G$.
\begin{lemma}\label{lem:conjclass}
	Let $u \in C_n$. The following (i), (ii) hold.
	\begin{enumerate}[(i)]
		\item
		$${\rm Cl}(u)=\begin{cases}
			\{ u \}\,\, &\text{if}\,\, u \in C_n[\tau],\\
			\{ u, u^s \}\, \, \text\,\, &\text{otherwise}. 
		\end{cases}
		$$
		In the second case, we have $u \neq u^s$. Moreover,
		$${\rm C}_G(u)=\begin{cases}
			G\,\, &\text{if}\,\, u \in C_n[\tau],\\
			C_n\, \, \text\,\, &\text{otherwise}. 
		\end{cases}
		$$
		\item  ${\rm Cl}(bu)=buC_n^\tau$ and
		${C}_G(bu)=C_n[\tau] \cup buC_n[\tau]$.
	\end{enumerate}
\end{lemma}
\begin{proof}
	Let $u \in C_n$.
	\begin{enumerate}[(i)]
		\item For any $v \in C_n$, we have $vuv^{-1}=u$ and $(bv)u(bv)^{-1}=u^s$. Then, ${\rm Cl}(u)=\{u,u^s\}$, where $u=u^s\Leftrightarrow u^{s-1}=1\Leftrightarrow u \in C_n[\tau]$. The statement on ${\rm C}_G(u)$ can be checked easily.
		\item For any $v \in C_n$, we have $v(bu)v^{-1}=buv^{s-1}$ and $(bv)(bu)(bv)^{-1}=bu(uv^{-1})^{s-1}$.   Then, ${\rm Cl}(bu)=buC_n^{s-1}=buC_n^{\tau}$. The result on ${\rm C}_{G}(bu)$ is straightforward.
	\end{enumerate}
\end{proof}

In order to use Lemma~\ref{lem:conjclass}, we need to label the conjugacy classes of $G$. To that end, we write
$C_n[\tau]=\{u_1, \dots, u_\tau\}
$ and
$$
C_n \setminus C_n[\tau] =\{v_1, \dots, v_{(n-\tau)/2},v_{(n-\tau)/2+1},\dots,v_{n-\tau} \},
$$ where $v_{(n-\tau)/2+i}=v_i^s$ for any $1\leq i \leq (n-\tau)/2$. Moreover, let $\{w_1, \dots, w_{\tau}\}$ be a left transversal of $C_n^\tau$ in $C_n$. We now label the conjugacy classes of $G$ as follows.
\begin{enumerate}
	\item For $1 \leq i \leq \tau$, let $X_i={\rm Cl}(u_i)=\{u_i\},$
	\item for $1 \leq i \leq (n-\tau)/2$, let $Y_i={\rm Cl}(v_i)=\{v_i,v_{(n-\tau)/2+i}\}=\{v_i, v_i^s\},$ and
	\item for $1 \leq i \leq \tau$, let
	$Z_i={\rm Cl}(bw_i)=bw_iC_n^\tau$. 
\end{enumerate}
In other words, the conjugacy classes of $G$ are given by $X_i$, $Y_i$, and $Z_i$.


\subsection{Proof of Theorem \ref{thm:main1}}
In this subsection, we establish the main result (Theorem \ref{thm:main1}), which gives the dimension of $T(G)$. Recall that $G=D_{n,s}$.
First, we compute the dimension of the centralizer algebra $\widetilde{T}(G)$.
\begin{lemma}\label{lem:dimtildeT}
	We have 
	$$
	{\dim}_{\mathbb{C}}\widetilde{T}(G)=\frac{n^2+3n\tau+4\tau^2}{2}.
	$$
\end{lemma}
\begin{proof}
	Recall that
	$
	{\dim}_{\mathbb{C}}\widetilde{T}(G)= \frac{1}{|G|}\sum_{g \in G}|{\rm C}_{G}(g)|^2.$ Then, by Lemma \ref{lem:conjclass}, we have
	\begin{align*}
		{\dim}_{\mathbb{C}} \widetilde{T}(G)=&\frac{1}{2n}\left( \sum_{u \in C_n[\tau]} |{\rm C}_{G}(u)|^2+ \sum_{u \in C_n \setminus C_n[\tau]} |{\rm C}_{G}(u)|^2+ \sum_{u \in C_n} |{\rm C}_{G}(bu)|^2 \right)\\
		=&\frac{1}{2n} \left( \sum_{u \in C_n[\tau]} (2n)^2 + \sum_{u \in C_n \setminus C_n[\tau]} n^2 + \sum_{u \in C_n}\left(2|C_n[\tau]|\right)^2 \right)\\
		=&\frac{1}{2n} \left( \sum_{u \in C_n[\tau]} 4n^2 + \sum_{u \in C_n \setminus C_n[\tau]} n^2 + \sum_{u \in C_n}4\tau^2 \right)\\
		=&\frac{1}{2n}\left( 4n^2\tau+n^2(n-\tau)+4n\tau^2 \right)\\
		=&\frac{n^2+3n\tau+4\tau^2}{2}.
	\end{align*}
\end{proof}
Next, we determine the dimension of $T_0(G)$.
\begin{lemma}\label{lem:dimTo}
	We have 
	$$
	{\dim}_{\mathbb{C}}T_0(G)=\frac{n^2+3n\tau+4\tau^2}{2}.
	$$
\end{lemma}
\begin{proof} Recall that
	$
	{\dim}_{\mathbb{C}}T_0(G)= |\{(U,V,W) \mid W \subset U \cdot V \}|,
	$ where $U,V,W$ are conjugacy classes of $G$. Now, consider the triple $(U,V,W)$ such that $W \subset U \cdot V$.
	\begin{enumerate}
		\item  If $U=X_i$ and $V=X_j$ for some $1 \leq i,j \leq \tau$, then we have $U \cdot V =X_i \cdot X_j=X_k$, where $u_k=u_iu_j$. Therefore, $W=X_k$. Hence, there are $n_1=\tau^2$  contributions in ${\dim}_{\mathbb{C}}T_0(G)$.
		\item  If $U=X_i$ and $V=Y_j$ for some $ 1 \leq i \leq \tau$ and $1 \leq j \leq (n-\tau)/2$, then we have $U \cdot V=X_i \cdot Y_j = Y_k$, where $v_k=u_iv_j$ or $v_k=u_iv_j^{s}$. Therefore, $W=Y_k$. Hence, there are $n_2=\tau(n-\tau)/2$ contributions in ${\dim}_{\mathbb{C}}T_0(G)$.
		\item If $U=Y_i$ and $V=X_j$ for some $1 \leq i \leq (n-\tau)/2$ and $1 \leq j \leq \tau$, then we have $Y_i \cdot X_j =Y_k$, where $v_k=u_jv_i$ or $v_k=u_jv_i^s$. Therefore, $W=Y_k$. Hence, there are $n_3=\tau(n-\tau)/2$ contributions in ${\dim}_{\mathbb{C}}T_0(G)$.
		\item If $U=X_i$ and $V=Z_j$ for some $1 \leq i,j \leq \tau$, then we have $U \cdot V=X_i \cdot Z_j=Z_k$, where $u_iw_j \in w_kC_n^\tau$. Therefore, $W=Z_k$. Hence, there are $n_4=\tau^2$ contributions in ${\dim}_{\mathbb{C}}T_0(G)$.
		\item  If $U=Z_i$ and $V=X_j$ for some $1 \leq i,j \leq \tau$, then we have $U \cdot V =Z_i \cdot X_j=Z_k$, where $u_jw_i \in w_kC_n^\tau$. Therefore, $W=Z_k$. Hence, there are $n_5=\tau^2$ contributions in ${\dim}_{\mathbb{C}}T_0(G)$.
		\item If $U=Y_i$ and $V=Y_j$ for some $1 \leq i,j \leq (n-\tau)/2$, then we have 
		$$U \cdot V=Y_i \cdot Y_j= \{ v_i v_j, (v_i v_j)^s \}\sqcup \{ v_i v_j^s,(v_i v_j^s)^s\}.$$ Here $\{ v_iv_j, (v_i v_j)^s \}$ or $\{ v_i v_j^s,(v_i v_j^s)^s\}$ are conjugacy classes of type $X_k$ or $Y_k$. Therefore, there are two choices for $W$. Consequently, there are $n_6=2((n-\tau)/2)^2$ contributions in ${\dim}_{\mathbb{C}}T_0(G)$.
		\item  If $U=Y_i$ and $V=Z_j$ for some $1 \leq i \leq (n-\tau)/2$ and $1 \leq j \leq \tau$ then we have $Y_i \cdot Z_j=Z_k$, where $v_i w_j \in w_kC_n^\tau$. Therefore, $W=Z_k$. So, there are $n_7=\tau(n-\tau)/2$ contributions in ${\dim}_{\mathbb{C}}T_0(G)$.
		\item  If $U=Z_i$ and $V=Y_j$ for some $1 \leq i \leq \tau$ and $1 \leq j \leq (n-\tau)/2$, then we have $Z_i \cdot Y_j = Z_k$, where $w_iv_j \in w_k C_n^\tau$. Therefore, $W=Z_k$. Hence, there are $n_8=\tau(n-\tau)/2$ contributions in ${\dim}_{\mathbb{C}}T_0(G)$.
		\item If $U=Z_i$ and $V=Z_j$ for some $1 \leq i, j \leq \tau$, then we have $Z_iZ_j=w_kC_n^\tau$, where $w_iw_j \in w_k C_n^\tau$. To count all contributing triples, for each $1 \leq k \leq \tau$, notice that there are $\tau$ couples $(Z_i,Z_j)$ such that $Z_i \cdot Z_j=w_kC_n^\tau$. Since  $\{w_kC_n^\tau \mid 1 \leq  k \leq \tau \}$ is a partition of $C_n$, the number of contributions in ${\dim}_{\mathbb{C}}T_0(G)$ is $n_9=\tau(\tau+(n-\tau)/2)$, where $\tau+(n-\tau)/2$ is the number of conjugacy classes (all $X_k$ and $Y_k$) inside $C_n$.
	\end{enumerate}
	Combining all the steps $1$ to $9$ above, we get that
	\begin{align*}
		{\dim}_{\mathbb{C}}T_0(G)=\sum_{i=1}^9n_i=\frac{n^2+3n\tau+4\tau^2}{2}.
	\end{align*}
		This completes the proof.
\end{proof}

In virtue of Lemma~\ref{lem:dimtildeT} and Lemma~\ref{lem:dimTo}, the proof of Theorem~\ref{thm:main1} is complete.


\section{Wedderburn decomposition for $T(D_{n,s})$}\label{sec:wedd}

In this section, we find the Wedderburn decomposition for the Terwilliger algebra of the group $D_{n,s}$. Let $\omega=e^{\frac{2 \pi {\bf i}}{n}}$.
The representation theory of the group $D_{n,s}$ is fairly straightforward. We refer the readers to \cite[Section~ 3.2.1]{behajaina2022integral} and \cite[Section~8.1]{serre} for more details.\\ Recall that $\tau={\rm gcd}(s-1,n)$. The irreducible representations of $D_{n,s}$ are as follows:

\begin{enumerate}[(1)]
	\item {\bf One-dimensional representations:} Let $0 \leq k \leq \tau-1$. The one-dimensional irreducible characters are one of the following:
	\begin{itemize}
		\item the morphism
		$
		\Psi_{k,1}:D_{n,s} \rightarrow \mathbb{C}^*
		$ such that $\Psi_{k,1}(a)=\omega^{\frac{nk}{\tau}}$
		and $\Psi_{k,1}(b)=1$,
		\item the morphism
		$
		\Psi_{k,2}:D_{n,s} \rightarrow \mathbb{C}^*
		$ such that $\Psi_{k,2}(a)=\omega^{\frac{nk}{\tau}}$
		and $\Psi_{k,2}(b)=-1$.
	\end{itemize}
	Denote by $\psi_{k,1}$ and $\psi_{k,2}$ the corresponding characters, respectively.
	\vskip 2mm
	\item {\bf Two-dimensional representations:} Let $1 \leq  k \leq n-1$ such that $k \not \equiv 0 \ (\operatorname{mod}{\frac{n}{\tau}})$. A two-dimensional irreducible character of $D_{n,s}$ is a morphism of the form
	$$
	\Phi_k : D_{n,s} \rightarrow \operatorname{GL}_2(\mathbb{C})
	$$ such that
	\begin{align*}
			\Phi_k(a)=\begin{pmatrix} 
			\omega^k &  0\\
			0 & \omega^{ks}
		\end{pmatrix}\qquad \mbox{and} \qquad
		\Phi_k(b)=\begin{pmatrix}
		0 & 1 \\
		1 & 0 
		\end{pmatrix}.
	\end{align*}

	 Denote by $\phi_k$ the corresponding character. Note that in this case $\Phi_k$ is equivalent to $\Phi_l$ if and only if $k \equiv ls \ (\operatorname{mod}{n})$.
\end{enumerate}
Next, we give the character table of $D_{n,s}$. First, we need to establish new notations for the representatives of the conjugacy classes. For $1\leq i\leq \tau$, we let $\ell_i := \frac{ni}{\tau}$. Thus, we may assume that $u_i = a^{\ell_i}$ for $1\leq i \leq \tau$. Then, we have
$$
C_n[\tau]= \left\langle a^{\frac{n}{\tau}} \right\rangle = \{ u_i \mid 1 \leq i \leq \tau \}.
$$
For $1\leq i \leq \frac{n-\tau}{2}$, we let $m_i$ be the smallest positive integer such that $v_i = a^{m_i}$. Finally, without loss of generality, since $\{ a,a^2,\ldots,a^{\tau-1},a^\tau \}$ is a left transversal for the subgroup $C_n^\tau = \langle a^\tau\rangle$, we may assume that $w_i=a^i$ for all $ 1 \leq i \leq \tau$.
Using these notations, the character table of $D_{n,s}$ is given by
\begin{table}[H]
	\centering
	\begin{tabular}{c||c|c|c}
		& $X_i={\rm Cl}(u_i)$ ($ 1 \leq i \leq \tau$) & $Y_i={\rm Cl}(v_i)$ ($1 \leq i \leq (n-\tau)/2)$ & $Z_i={\rm Cl}(b w_i)$ ($1 \leq i \leq \tau$)
		\\
		\hline \hline
		$\psi_{k,1}$& $\omega^{\frac{nk\ell_i}{\tau}}$ & $\omega^{\frac{nkm_i}{\tau}}$ & $\omega^{\frac{nki}{\tau}}$
		\\
		\hline
		$\psi_{k,2}$& $\omega^{\frac{nk\ell_i}{\tau}}$ & $\omega^{\frac{nkm_i}{\tau}}$ & $-\omega^{\frac{nki}{\tau}}$
		\\
		\hline
		$\phi_k$& $\omega^{k\ell_i}+\omega^{k\ell_is}$ & $\omega^{km_i}+\omega^{km_is}$ & $0$\\
	\end{tabular} 
	\caption{Character table of $D_{n,s}$.}
\end{table}
We now consider the decomposition
\begin{equation}\label{eq:decpi}
	\pi=\sum_{k=0}^{\tau-1}d_{k,1}\psi_{k,1}+\sum_{k=0}^{\tau-1}d_{k,2}\psi_{k,2}+\sum_{k=1\atop k \sim ks}^{n-1}d_k\phi_k ,
\end{equation} where $\pi$ is the permutation representation of $D_{n,s}$ acting on itself by conjugation as in Section~\ref{subsec:weddecth}. In \eqref{eq:decpi}, for $1 \leq k \leq n-1$, since $\Phi_k$ is equivalent to $\Phi_{ks}$, only one of $d_k$ or $d_{ks}$ is considered, this explains the notation $k \sim ks$. 

Next, we compute the coefficients $d_{k,1}$, $d_{k,2}$, and $d_k$ in (\ref{eq:decpi}). Let $r=\omega^{\frac{n}{\tau}}$ (i.e. a primitive $\tau$-th root of unity) and $ 0 \leq k \leq \tau-1$. We have
\begin{align*}
	d_{k,1}=&\sum_{i=1}^{\tau}\overline{\psi_{k,1}(u_i)}+\sum_{i=1}^{\frac{(n-\tau)}{2}}\overline{\psi_{k,1}(v_i)}+\sum_{i=1}^{\tau}\overline{\psi_{k,1}(bw_i)}\\
	=&\sum_{i=1}^\tau \omega^{-\frac{nk\ell_i}{\tau}}+\sum_{i=1}^{\frac{(n-\tau)}{2}} \omega^{-\frac{nkm_i}{\tau}}+\sum_{i=1}^\tau\omega^{-\frac{nki}{\tau}}
	\\
	=&\frac{1}{2}\sum_{i=1}^\tau \omega^{-\frac{nk\ell_i}{\tau}} + \frac{1}{2}\sum_{i=0}^{n-1} \omega^{-\frac{nki}{\tau}}+\sum_{i=1}^\tau\omega^{-\frac{nki}{\tau}}
	\\
	= & \frac{1}{2}\sum_{i=1}^\tau\left( r^{-\frac{nk}{\tau}} \right)^i+\frac{1}{2}  \sum_{i=0}^{n-1}\left( \omega^{-\frac{nk}{\tau}} \right)^i+ \sum_{i=1}^\tau r^{-ki}.
\end{align*}

\begin{itemize}
	\item If $\tau$ divides $k$, then 
	$d_{k,1}= \frac{\tau}{2} + \frac{n}{2}+\tau=\frac{n+3\tau}{2}.
	$
	\item If $\tau$ does not divide $k$, then:
	\begin{itemize}
		\item If $\tau$ divides $nk/\tau$, then
		$
		d_{k,1}=\frac{\tau}{2}+0+0=\frac{\tau}{2}.
		$
		\item If $\tau$ does not divide $nk/\tau$, then
		$
		d_{k,1}=0+0+0=0.
		$
	\end{itemize}
\end{itemize}
Consequently, we have
\begin{align}
	d_{k,1}=
	\begin{cases}
		\frac{n+3\tau}{2} & \textrm{if}\,\, k \equiv 0 \ (\operatorname{mod}{\tau})\\
		\frac{\tau}{2} & \textrm{if}\,\, k \not \equiv 0 \ (\operatorname{mod}{\tau})\,\, \textrm{and}\,\, nk \equiv 0 \ (\operatorname{mod}{\tau^2})\\
		0 & \textrm{otherwise.}
	\end{cases}\label{eq:first}
\end{align}
Using the same argument, for $0\leq k\leq \tau-1$ we have 
\begin{align}
	d_{k,2}=
	\begin{cases}
		\frac{n-\tau}{2} & \textrm{if}\,\, k \equiv 0 \ (\operatorname{mod}{\tau})\\
		\frac{\tau}{2} & \textrm{if}\,\, k \not \equiv 0 \ (\operatorname{mod}{\tau})\,\, \textrm{and}\,\, nk \equiv 0 \ (\operatorname{mod}{\tau^2})\\
		0 & \textrm{otherwise.}
	\end{cases}\label{eq:second}
\end{align}

Next, we compute the coefficient $d_k$. Let $ 1 \leq k \leq n-1 $ such that $k \not \equiv 0 \ (\operatorname{mod}{n/\tau})$. Then,
\begin{align*}
	d_{k}= &\sum_{i=1}^{\tau}\overline{\phi_k(u_i)}+\sum_{i=1}^{\frac{(n-\tau)}{2}}\overline{\phi_k(v_i)}+\sum_{i=1}^{\tau}\overline{\phi_k(bw_i)}\\
	= &\sum_{i=1}^{\tau}\left( \omega^{-k\ell_i}+\omega^{-k\ell_is} \right)+\sum_{i=1}^{\frac{(n-\tau)}{2}}\left( \omega^{-km_i}+\omega^{-km_is}\right)\\
	= &\frac{1}{2}\sum_{i=1}^\tau\left( r^{-ki}+r^{-kis}\right)+\frac{1}{2}\sum_{i =0}^{n-1}\left( \omega^{-ki}+\omega^{-kis} \right).
\end{align*}
Since $k \not \equiv 0 \ (\operatorname{mod}{n})$, we have
$
d_k=\frac{1}{2}\sum_{i=1}^\tau\left( r^{-ki}+r^{-kis}\right).
$ Consequently,

\begin{align}
	d_k=
	\begin{cases}
		\tau \quad \textrm{if} \,\, k\equiv0 \ (\operatorname{mod}{\tau})\\
		0\quad \textrm{otherwise}.
	\end{cases}\label{eq:third}
\end{align}
Now we are ready to determine the Wedderburn decomposition of $T(D_{n,s})$.
\begin{theorem}\label{weddecomdns}
	If $M_d$ is the $d$-dimensional complex matrix algebra, then
	\begin{align*}
	T(D_{n,s})\cong\left(\bigoplus_{k=0}^{\tau-1}M_{d_{k,1}}\right) \bigoplus \left(\bigoplus_{k=0}^{\tau-1}M_{d_{k,2}}\right) \bigoplus \left(\bigoplus_{k=1\atop k\sim ks}^{n-1}M_{d_{k}}\right),
	\end{align*}
	where $d_{k,1},\ d_{k,2}$, and $d_k$ are respectively defined in \eqref{eq:first}, \eqref{eq:second}, and \eqref{eq:third}.
\end{theorem}
\begin{proof}
	The proof immediately follows from the above computations and Section~\ref{subsec:weddecth}.
\end{proof}

As a particular case, for the dihedral group $D_{n,-1}$ we get the following corollaries.
\begin{corollary}
	Let $n$ be even. If $M_d$ is the $d$-dimensional complex matrix algebra, then the Wedderburn decomposition of $T(D_{n,-1})$ is given by
	\begin{align*}
		\begin{cases}
			M_{a_n} \oplus M_{a_n-4} \oplus \underbrace{M_2 \oplus \ldots \oplus M_2}_{\lfloor \frac{n-1}{4} \rfloor} \hspace{2cm}&\mbox{ if } n \equiv 2 \ (\operatorname{mod}4)\\
			M_{a_n} \oplus M_{a_n-4} \oplus M_1 \oplus M_1 \oplus \underbrace{M_2 \oplus \ldots \oplus M_2}_{\lfloor \frac{n-1}{4} \rfloor -1} \hspace{2cm}&\mbox{ if } n \equiv 0 \ (\operatorname{mod}4),
		\end{cases}
	\end{align*}
	where $a_n = \lfloor\frac{n-1}{2}\rfloor +4$.
	\label{thm:main2}
\end{corollary}

\begin{corollary}
	Let $n$ be odd. If $M_d$ is the $d$-dimensional complex matrix algebra, then the Wedderburn decomposition of $T(D_{n,-1})$ is given by
	\begin{align*}
		M_{b_n} \oplus M_{b_n-2} \oplus \underbrace{M_1 \oplus \ldots \oplus M_1}_{\lfloor \frac{n-1}{4} \rfloor} ,
	\end{align*}
	where $b_n = \lfloor\frac{n-1}{2}\rfloor +2$.
	\label{thm:main3}
\end{corollary} 

\subsection*{Acknowledgment}
The research of the author is  supported in part by the Ministry of Education, Science and Sport of Republic of Slovenia (University of Primorska Developmental funding pillar).
The author would also like to thank the anonymous referees for their insightful comments.



\begin{thebibliography}{10}
	
	\bibitem{balamaceda1994terwilliger}
	J.~M.~P. Balamaceda and M.~Oura.
	\newblock The {T}erwilliger algebras of the group association schemes of $S_5$ and
	$A_5$.
	\newblock {\em Kyushu Journal of Mathematics}, 48(2):221--231, 1994.
	
	\bibitem{bannai1991subschemes}
	E.~Bannai.
	\newblock Subschemes of some association schemes.
	\newblock {\em J. Algebra}, 144(1):167--188, 1991.
	
	\bibitem{bannai1995terwilliger}
	E.~Bannai and A.~Munemasa.
	\newblock The {T}erwilliger algebras of group association schemes.
	\newblock {\em Kyushu Journal of Mathematics}, 49(1):93--102, 1995.
	
	\bibitem{bastian2021terwilliger}
	N.~L. Bastian.
	\newblock  Terwilliger Algebras for Several Finite Groups.
	\newblock PhD thesis, Brigham Young University, 2021.
	
	
	\bibitem{behajaina2022integral}
	A.~Behajaina and F.~Legrand.
	\newblock On integral mixed Cayley graphs over non-abelian finite groups admitting an abelian subgroup of index 2.
	\newblock {\em Linear Algebra Appl.}, 675:256-273, 2023.
	
	
	\bibitem{bose1952classification}
	R.~C. Bose and T.~Shimamoto.
	\newblock Classification and analysis of partially balanced incomplete block
	designs with two associate classes.
	\newblock {\em J. Am. Stat. Assoc.}, {\bf
		47}(258):151--184, 1952.
	
	\bibitem{delsarte1998association}
	P.~Delsarte and V.~I. Levenshtein.
	\newblock Association schemes and coding theory.
	\newblock {\em IEEE Transactions on Information Theory}, 44(6):2477--2504,
	1998.
	
	\bibitem{godsil2017algebraic}
	C.~D. Godsil.
	\newblock  Algebraic Combinatorics.
	\newblock Routledge, 2017.
	
	\bibitem{hamid2019terwilliger}
	N.~Hamid and M.~Oura.
	\newblock Terwilliger algebras of some group association schemes.
	\newblock {\em Mathematical Journal of Okayama University}, 61(1):199--204,
	2019.
	
	\bibitem{martin2009commutative}
	W.~J. Martin and H.~Tanaka.
	\newblock Commutative association schemes.
	\newblock {\em Eur. J. Comb.}, 30(6):1497--1525, 2009.
	
	\bibitem{serre}
	J.-P. Serre.
	\newblock Repr{\'e}sentation lin{\'e}aire des groupes finis.
	\newblock {\em Hermann, Paris, 1971.}
	
	\bibitem{PTerwilliger}
	P.~Terwilliger.
	\newblock The subconstituent algebra of an association scheme, I, II, III.
	\newblock {\em J. Alg. Combin, 1 (1992), 363–388, 2 (1993),
		73–103, 2 (1993), 177–210}.
	
	\bibitem{zieschang2006algebraic}
	P.-H. Zieschang.
	\newblock An algebraic approach to association schemes.
	\newblock{\em  Springer, 2006}.
	
	\bibitem{zieschang2005theory}
	P.-H. Zieschang.
	\newblock  Theory of association schemes.
	\newblock {\em Springer Science} \& {\em Business Media, 2005}.
	

\end{thebibliography}
\end{document}